\documentclass[11pt, a4paper]{amsart}
\usepackage[top=35truemm,bottom=35truemm,left=35truemm,right=35truemm]{geometry}
\usepackage{amsmath,amsthm,amssymb}
\usepackage[dvipdfmx]{graphicx}
\usepackage{braket}
\usepackage{color}

\numberwithin{equation}{section}
\numberwithin{figure}{section}

\title{Coxeter systems with $2$-dimensional Davis complexes, growth rates and Perron numbers}
\author{Naomi Bredon}
\address{Department of Mathematics, University of Fribourg, Switzerland}
\email{naomi.bredon@unifr.ch}
\author{Tomoshige Yukita}
\address{Department of Mathematics, School of Education, Waseda University, Japan}
\email{yshigetomo@suou.waseda.jp}

\newtheorem{thm}{Theorem}[section]
\newtheorem*{thm*}{Theorem}
\newtheorem{ex}{Example}
\newtheorem{lem}[thm]{Lemma}
\newtheorem{cor}[thm]{Corollary}
\newtheorem*{conj}{Conjecture}

\theoremstyle{definition}

\newtheorem{rem}{Remark}
\newtheorem{theorem}{Theorem}

\newcommand{\grouppresentation}[2]{\left\langle #1 \middle| #2 \right\rangle}
\newcommand{\abs}[1]{\left\lvert #1 \right\rvert}

\DeclareMathOperator{\Cox}{Cox}
\DeclareMathOperator{\Isom}{Isom}

\def\S{\mathbb{S}}
\def\E{\mathbb{E}}
\def\H{\mathbb{H}}
\def\X{\mathbb{X}}
\def\N{\mathbb{N}}
\def\R{\mathbb{R}}

\begin{document}

\begin{abstract}
In this paper,
we study growth rates of Coxeter systems with Davis complexes of dimension at most $2$.
We show that if the Euler characteristic $\chi$ of the nerve of a Coxeter system is vanishing (resp. positive),
then its growth rate is a Salem (resp. a Pisot) number.
In this way, we extend results due to Floyd and Parry \cite{Floyd92, Parry93}.
Moreover,
in the case where $\chi$ is negative,
we provide infinitely many non-hyperbolic Coxeter systems whose growth rates are Perron numbers.
\end{abstract}

\maketitle

\section{Introduction}
Let $\Gamma$ be a finitely generated group with generating set $S$.
For an element $x\in \Gamma$,
we write $\abs{x}_S$ for the word length with respect to $S$.
The growth rate of $(\Gamma, S)$ is defined by 
\begin{equation*}
\tau(\Gamma, S)=\limsup_{\ell\to \infty}\sqrt[\ell]{a_\ell} \ ,
\end{equation*}
where $a_\ell$ is the number of elements of $\Gamma$ of word length $\ell$.
Gromov's polynomial growth theorem \cite{Gromov81} states that $\Gamma$ has a nilpotent subgroup of finite index if and only if there exists positive constants $C>0$ and $d>0$ such that $a_\ell\leq C k^d$ for $\ell\geq 0$.
If $(\Gamma, S)$ satisfies the latter property,
then we say that $(\Gamma, S)$ has \emph{polynomial growth}.
In this case,
one has $\tau(\Gamma, S)=1$.
The pair $(\Gamma, S)$ is said to have \emph{exponential growth} when $\tau(\Gamma, S)>1$.
Note that there exist pairs of groups and finite generating sets which have neither polynomial growth nor exponential growth (see \cite{Grigorchuk84} for example).

\vspace{2mm}
Suppose that $(\Gamma, S)$ is an abstract Coxeter system,
that is,
$\Gamma$ is generated by $S$ and has the following presentation
\begin{equation*}
\Gamma=\grouppresentation{s_1,\ldots,s_N \ }{ \ (s_is_j)^{k_{ij}}\text{ for }1\leq i,j\leq N},
\end{equation*}
where $k_{ii}=1$ and $k_{ij}\geq 2$ (see Section \ref{Section2.1}).
There are three types of Coxeter systems:
spherical, affine, and otherwise.
If $(\Gamma, S)$ is spherical or affine,
then it has polynomial growth.
Therefore,
our interest lies in the growth rates of non-spherical, non-affine Coxeter systems.
Typical examples of such Coxeter systems are cofinite hyperbolic Coxeter systems (see Section \ref{Section2.2}).

\vspace{2mm}
In the study of the growth rates of hyperbolic Coxeter systems,
three kinds of real algebraic integers appear:
Salem numbers, Pisot numbers, and Perron numbers (see Section \ref{Section2.3}).
By results of Parry \cite{Parry93},
the growth rates of $2$- and $3$-dimensional cocompact hyperbolic Coxeter systems are Salem numbers.
Floyd showed that the growth rates of $2$-dimensional cofinite hyperbolic Coxeter systems are Pisot numbers \cite{Floyd92}.
Moreover,
their growth rates are limits of growth rates of $2$-dimensional cocompact hyperbolic Coxeter systems.
The second author proved that the growth rates of $3$-dimensional cofinite hyperbolic Coxeter systems are Perron numbers \cite{Yukita17, Yukita18}.
Kolpakov proved that the growth rates of particular $3$-dimensional cofinite hyperbolic Coxeter systems are Pisot numbers \cite{Kolpakov12}.
With all the above considerations,
we are interested in the relation between the geometric properties of Coxeter systems and the arithmetic nature of their growth rates as follows.

\vspace{2mm}
Let $(\Gamma, S)$ be an abstract Coxeter system.
Its \textit{nerve} $L(\Gamma, S)$ is the abstract simplicial complex defined as follows (see Section \ref{Section2.2}).
The vertex set is $S$.
For a nonempty subset $T=\{s_{i_1},\ldots,s_{i_n}\}\subset S$,
the vertices $s_{i_1},\ldots,s_{i_n}$ span an $(n-1)$-simplex if and only if $T$ generates a finite subgroup of $\Gamma$.
By abuse of notation, we write $L(\Gamma, S)$ for its geometric realization (see \cite[Chapter 1, Section 3]{Munkres84} for details).
The \textit{dimension of $L(\Gamma, S)$} is defined as the maximal dimension of simplices of $L(\Gamma, S)$.
It coincides with the maximal cardinality $-1$ of a subset of $S$ which generates a finite subgroup of $\Gamma$.
The dimension of $(\Gamma, S)$ is defined as $\dim{L(\Gamma, S)}+1$,
it is the dimension of the Davis complex of $(\Gamma, S)$; for details see \cite{Davis12, FeliksonTumarkin10}.

\vspace{2mm}
In this paper,
we study the arithmetic nature of the growth rates of non-spherical, non-affine Coxeter systems $(\Gamma, S)$ of dimension at most $2$.
We will prove the following theorems.
\begin{theorem}\label{theorem:A}
If $\chi(L(\Gamma, S))=0$,
then the growth rate $\tau(\Gamma, S)$ is a Salem number.
\end{theorem}
\begin{theorem}\label{theorem:B}
If $\chi(L(\Gamma, S))\geq 1$,
then the growth rate $\tau(\Gamma, S)$ is a Pisot number.
Moreover,
there exists a sequence of Coxeter systems $(\Gamma_n, S_n)$ with vanishing Euler characteristic such that the growth rate $\tau(\Gamma_n, S_n)$ converges to $\tau(\Gamma, S)$ from below.
\end{theorem}
\begin{theorem}\label{theorem:C}
There exists infinitely many non-hyperbolic Coxeter systems $(\Gamma, S)$ with $\chi(L(\Gamma, S))\leq -1$ whose growth rates are Perron numbers.
\end{theorem}

This paper is organized as follows.
In Section \ref{Section2},
we provide the necessary background about Coxeter systems, their nerves, and their growth rates.
Theorem \ref{theorem:A} is discussed in Section \ref{Section3} where we consider Coxeter systems with vanishing Euler characteristic.
This extends the result by Parry \cite{Parry93}.
Section \ref{Section4} is devoted to the study of Coxeter systems with positive Euler characteristic where we prove Theorem \ref{theorem:B} generalizing Floyd's result \cite{Floyd92}.
In Section \ref{Section5},
we consider infinite sequences of Coxeter systems with negative Euler characteristic and prove in Theorem \ref{thm:5.1} that their growth rates are Perron numbers.
In this way,
we obtain Theorem \ref{theorem:C}.

\section{Preliminaries}\label{Section2}

\subsection{Coxeter systems}\label{Section2.1}
For a group $\Gamma$ with generating set $S=\{s_1,\ldots,s_N\}$,
the pair $(\Gamma, S)$ is called a \textit{Coxeter system} if $\Gamma$ has the presentation 
\begin{equation*}
\Gamma=\grouppresentation{s_1,\ldots,s_N \ }{ \ (s_is_j)^{k_{ij}}\text{ for }1\leq i,j\leq N},
\end{equation*}
where $k_{ii}=1$ and $k_{ij}\geq 2$.
In the case where $s_is_j$ has infinite order,
we put $k_{ij}=\infty$.
The \textit{rank} of a Coxeter system $(\Gamma, S)$ is defined as the cardinality $\#S$ of $S$.
For a subset $T\subset S$,
the subgroup $\Gamma_T$ of $\Gamma$ generated by $T$ is called a \textit{parabolic subgroup} of $\Gamma$, with $\Gamma_\emptyset=\{1\}$ by convention.

\vspace{2mm}
Given a Coxeter system $(\Gamma, S)$ of rank $N$ define the \textit{cosine matrix} $C_{(\Gamma, S)}=(c_{ij})\in M_N(\R)$ associated to $(\Gamma, S)$ as follows:
\begin{equation*}
c_{ij}=\begin{cases}
-\cos{\dfrac{\pi}{k_{ij}}} & \text{if }k_{ij}<\infty \ ,\\
-1 & \text{if }k_{ij}=\infty \ .
\end{cases}
\end{equation*}
The Coxeter system $(\Gamma, S)$ is said to be \textit{spherical} (resp. \textit{affine}),
if $C_{(\Gamma, S)}$ is positive definite (resp. positive semidefinite).

\vspace{2mm}
In this paper,
a graph $X$ is said to be \textit{simple} if $X$ has no loops and multiple edges.
We associate to a Coxeter system $(\Gamma, S)$ two kinds of edge-labeled simple graphs:
the \textit{Coxeter diagram} $\Cox{(\Gamma, S)}$ and the \textit{presentation diagram} $X(\Gamma, S)$.
The Coxeter diagram $\Cox{(\Gamma, S)}$ is defined as follows.
The vertex set is $S$.
Two vertices $s_i$ and $s_j$ are connected by an edge if and only if $k_{ij}\geq 3$.
The edge between $s_i$ and $s_j$ is labeled by $k_{ij}$ if $k_{ij}\in{\{4,5,\ldots\}\cup{\{\infty\}}}$.
A Coxeter system $(\Gamma, S)$ is said to be \textit{irreducible} if the underlying graph of $\Cox{(\Gamma, S)}$ is connected.
It is known that a spherical (resp. affine) Coxeter system decomposes into a direct product of irreducible spherical (resp. spherical and affine) Coxeter systems.
The Coxeter diagrams of irreducible spherical and affine Coxeter systems are depicted in Figure \ref{fig:figure2.1} and Figure \ref{fig:figure2.2}, respectively (see \cite[p.32, p.34]{Humphreys90}).
\begin{figure}[htbp]
\centering
\includegraphics[scale=0.5]{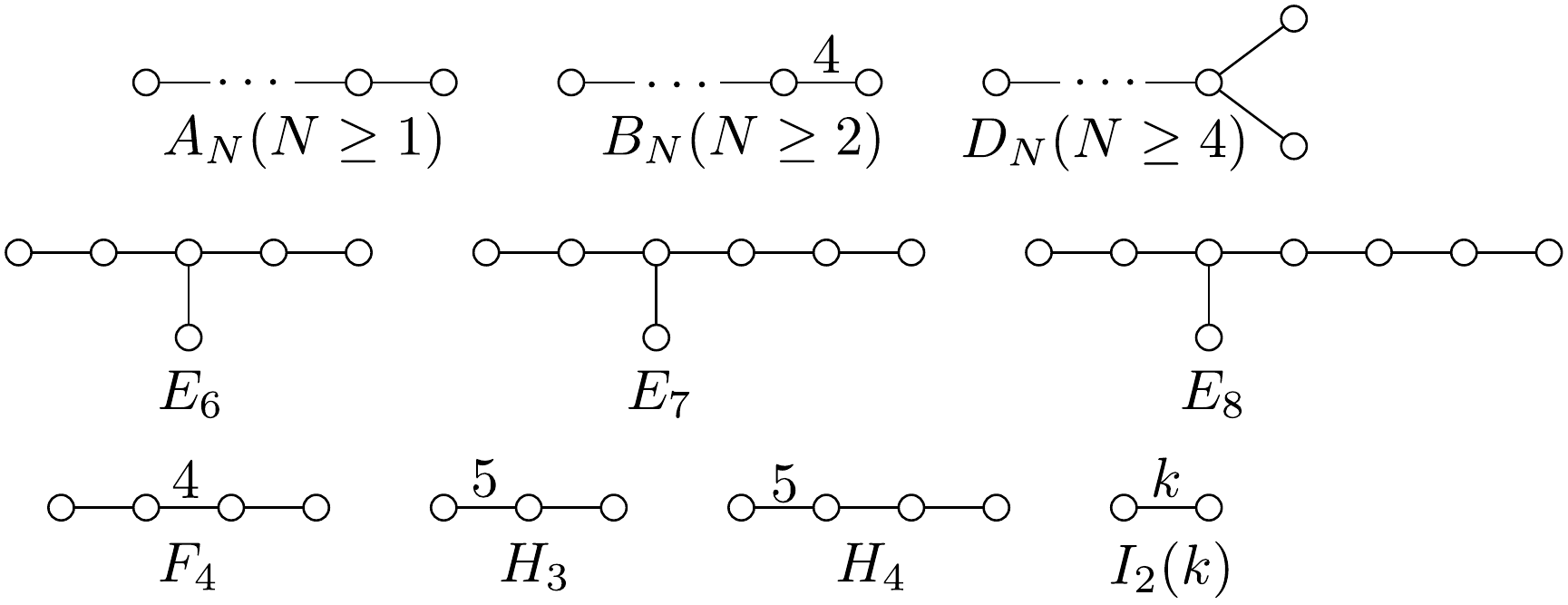}
\caption{Irreducible spherical Coxeter systems of rank $N$}
\label{fig:figure2.1}
\end{figure}
\begin{figure}[htbp]
\centering
\includegraphics[scale=0.5]{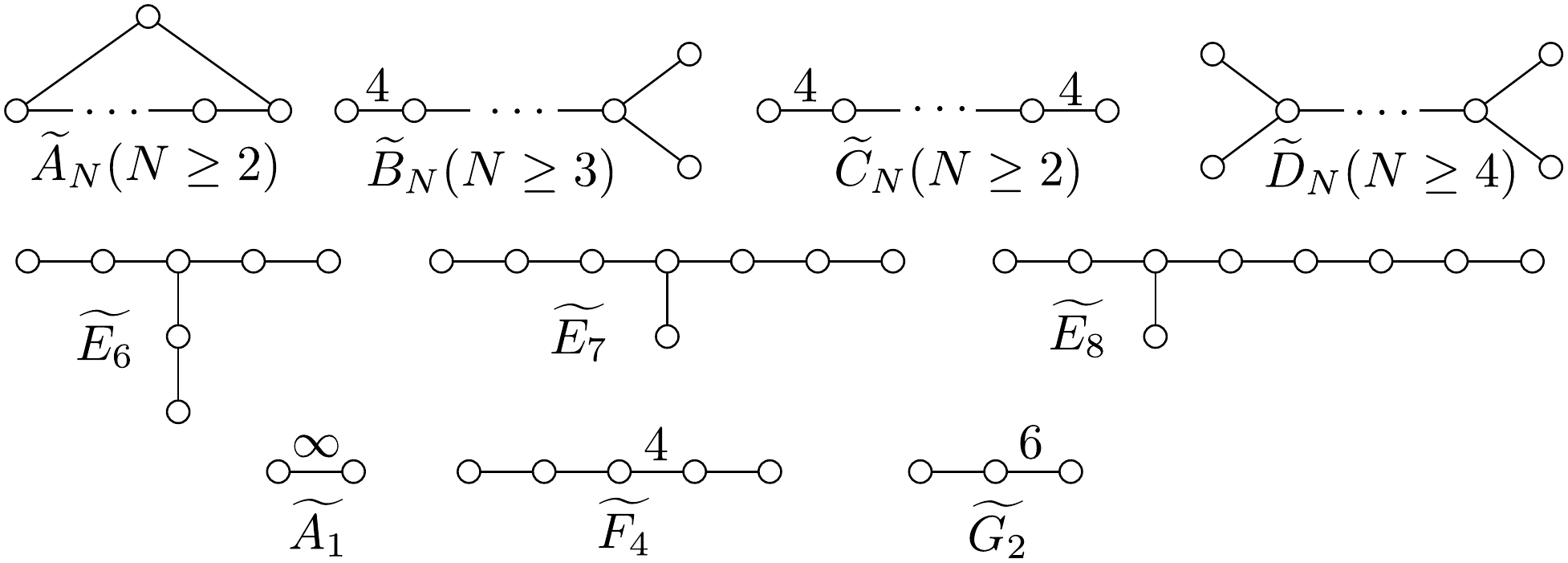}
\caption{Irreducible affine Coxeter systems of rank $N+1$}
\label{fig:figure2.2}
\end{figure}

\noindent
Let $(\Gamma, S)$ be a Coxeter system.
The \textit{presentation diagram $X(\Gamma, S)$} is defined as follows.
The vertex set is $S$.
Two vertices $s_i$ and $s_j$ are connected by an edge labeled by $k_{ij}$ when $k_{ij}<\infty$.
It follows that the underlying graphs of the presentation diagrams of spherical Coxeter systems of rank $N$ are complete graphs with $N$ vertices.
For example,
Figure \ref{fig:figure2.3} shows the presentation diagrams of the spherical Coxeter systems of rank $3$.
\begin{figure}[htbp]
\centering
\includegraphics[scale=0.5]{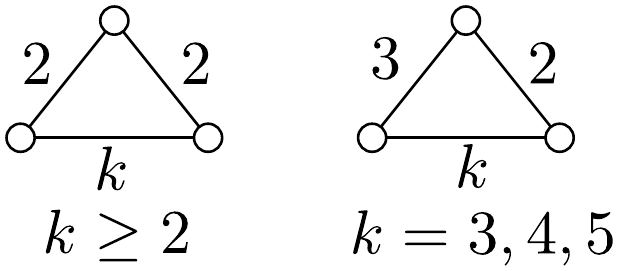}
\caption{The presentation diagrams of the spherical Coxeter systems of rank $3$}
\label{fig:figure2.3}
\end{figure}

\subsection{Geometric Coxeter groups and nerves}\label{Section2.2}
For more details about geometric Coxeter groups and nerves of Coxeter systems we refer to \cite{Davis12, Ratcliffe19}.

\vspace{2mm}
Let us denote by $\X^n$ the $n$-dimensional spherical space $\S^n$, Euclidean space $\E^n$, or hyperbolic space $\H^n$.
An \textit{$n$-dimensional Coxeter polytope} $P\subset \X^n$ is the intersection of finitely many half-spaces whose interior is nonempty and dihedral angles are of the form $\pi/\mathord{k}$ for $k\geq 2$ or equal to zero.
Given an $n$-dimensional Coxeter polytope $P\subset \X^n$ the set $S_P$ of the reflections in the bounding hyperplanes of $P$ generates a discrete subgroup $\Gamma_P$ of $\Isom{(\X^n)}$.
The pair $(\Gamma_P, S_P)$ is a Coxeter system, and is called an \textit{$n$-dimensional geometric Coxeter system associated with $P$}.
The group $\Gamma_P$ is called the \textit{$n$-dimensional geometric Coxeter group associated with $P$}.
It is known that $P$ is a fundamental polytope for $\Gamma_P$ and the orbit $\Set{gP|g\in \Gamma_P}$ of $P$ gives rise to an exact tessellation of $\X^n$.
Furthermore,
$\Gamma_P$ is said to be \textit{cocompact} (resp. \textit{cofinite})
when $P$ is compact (resp. not compact but of finite volume).
For a hyperbolic Coxeter polytope $P$,
we say that $\Gamma_P$ is \textit{ideal}
when every vertex of $P$ lies on the boundary at infinity $\partial{\H^n}$.
For each irreducible spherical (resp. affine) Coxeter system $(\Gamma, S)$,
there exists a spherical (resp. compact Euclidean) Coxeter polytope $P$ such that $(\Gamma, S)=(\Gamma_P, S_P)$.
Therefore,
if $(\Gamma, S)$ is a spherical (resp. affine) Coxeter system,
then $\Gamma$ is finite (resp. virtually nilpotent).
In contrast to this,
if $(\Gamma, S)$ is non-spherical and non-affine,
then $\Gamma$ contains a free group of rank at least $2$; see \cite{Harpe87}.

\vspace{2mm}
Let $(\Gamma, S)$ be an abstract Coxeter system.
The \textit{nerve} $L(\Gamma, S)$ is an abstract simplicial complex defined as follows.
The vertex set is $S$,
and for a non-empty subset $T=\{s_{i_1},\ldots,s_{i_n}\}\subset S$,
the vertices $s_{i_1},\ldots,s_{i_n}$ span an $(n-1)$-simplex if and only if the parabolic subgroup $\Gamma_T$ is finite.
For simplicity of notation,
we continue to write $L(\Gamma, S)$ for its geometric realization (see \cite[Chapter 1, Section 3]{Munkres84} for details).
If the maximal rank for spherical parabolic subgroups equals $d+1$,
then we say that $\dim{L(\Gamma, S)}=d$. 
The \textit{dimension of $(\Gamma, S)$}, denote by $\dim{(\Gamma, S)}$, is defined as $\dim{L(\Gamma, S)}+1$.
It is the dimension of the Davis complex of $(\Gamma, S)$; see \cite{Davis12, FeliksonTumarkin10}. 

\vspace{2mm}
In this paper,
we consider Coxeter systems of dimension at most $2$.
In particular,
such a class of Coxeter systems contains hyperbolic Coxeter groups of dimension $2$ and ideal hyperbolic Coxeter groups of dimension $3$.
Indeed,
for such groups,
maximal spherical subgroups are of rank at most $2$.
For a Coxeter system $(\Gamma, S)$ of dimension at most $2$,
it is easy to see that the underlying graph of $X(\Gamma, S)$ is the geometric realization of the nerve $L(\Gamma, S)$.
Therefore the Euler characteristic $\chi(L(\Gamma, S))$ equals the one of the underlying graph of $X(\Gamma, S)$.
It is known that the Euler characteristic of a graph is the number of vertices minus the number of edges.

\subsection{Growth rates of Coxeter systems}\label{Section2.3}
Let $(\Gamma, S)$ be a Coxeter system.
For $x\in \Gamma$,
we define its \textit{word length with respect to $S$} by
\begin{equation*}
\abs{x}_S=\min{\Set{n\in \N|x=s_1\cdots s_n \ (s_1,\ldots,s_n\in S)}}.
\end{equation*}
By convention $\abs{1}_S=0$.
The \textit{growth series $f_{(\Gamma, S)}(z)$} of $(\Gamma, S)$ is defined by
\begin{equation*}
f_{(\Gamma, S)}(z)=\sum_{\ell\geq 0} a_\ell z^\ell,
\end{equation*}
where $a_\ell$ is the number of the elements of $\Gamma$ of word length $\ell$.
If $(\Gamma, S)$ is spherical,
then $f_{(\Gamma, S)}(z)$ is a polynomial and called the \textit{growth polynomial} of $(\Gamma, S)$.

\vspace{2mm}
By a result of Solomon \cite{Solomon66},
the growth polynomials of spherical Coxeter systems can be computed in terms of its exponents.
For the list of exponents, see \cite{Humphreys90}.
For example,
the exponents of $A_N$ are given by $1,2,\ldots,N-1$,
and those of $I_2(k)$ are $1, k-1$.
For postive integers $m,m_1,\ldots,m_r$,
we put
\begin{equation*}
[m]=1+z+\cdots+z^{m-1}\quad \text{and}\quad  [m_1,\ldots, m_r]=[m_1]\cdots[m_r].
\end{equation*}
Solomon's formula states that for a spherical Coxeter system $(\Gamma, S)$ with the exponents $m_1,\ldots,m_r$, one has $f_{(\Gamma, S)}(z)=[m_1+1,\ldots,m_r+1]$.

\vspace{2mm}
If $(\Gamma, S)$ is non-spherical,
then the inverse of the radius of convergence of $f_{(\Gamma, S)}(z)$ is called the \textit{growth rate} of $(\Gamma, S)$, denoted by $\tau(\Gamma, S)$.
The Cauchy-Hadamard formula gives
\begin{equation*}
\tau(\Gamma, S)=\limsup_{\ell\to \infty}\sqrt[\ell]{a_\ell} \ .
\end{equation*}
If $(\Gamma, S)$ is affine,
the growth rate $\tau{(\Gamma, S)}$ is equal to $1$ by Gromov's polynomial growth theorem \cite{Gromov81}.
The following formula,
established by Steinberg,
is an important tool to compute the growth series of Coxeter systems.
\begin{thm}[Steinberg's formula \cite{Steinberg68}]
Let $(\Gamma, S)$ be a Coxeter system.

The following identity holds for the growth series $f_{(\Gamma, S)}(z)$. 
\begin{equation}
\dfrac{1}{f_{(\Gamma, S)}(z^{-1})}=\sum_{\substack{T\subset S \\ \#\Gamma_T<\infty}}\dfrac{(-1)^{\#T}}{f_{(\Gamma_T, T)}(z)} \ .\label{eq:2-1}
\end{equation}
\end{thm}
Steinberg's formula implies that the growth series is a rational function and satisfies that  
\begin{equation*}
\dfrac{1}{f_{(\Gamma, S)}(z^{-1})}=\dfrac{P(z)}{Q(z)} \ ,
\end{equation*}
where $P(z)$ and $Q(z)$ are monic polynomials with integer coefficients.
It follows that the growth rate $\tau(\Gamma, S)$ is the real root of $P(z)$ whose modulus is maximal among the roots of $P(z)$,
and hence $\tau(\Gamma, S)\geq 1$ is a real algebraic integer.

\begin{ex}\label{ex:1}
Consider the abstract Coxeter system $(\Gamma_\star,S_\star)$ whose presentation diagram is depicted in Figure \ref{fig:figure2.4}.
The spherical subgroups are $A_1$ and $A_2$, both with multiplicity four.
By Steinberg's formula \eqref{eq:2-1}, we compute its growth series:
 \begin{equation*}
\frac{1}{f_{(\Gamma_\star, S_\star)}(z^{-1})}=1-\frac{4}{[2]}+\frac{4}{[2,3]}=\frac{[2,3]-4[3]+4}{[2,3]}.
\end{equation*}
We write $P(z)$ for the numerator of $\frac{1}{f_{(\Gamma_\star, S_\star)}(z^{-1})}$, that is, 
\begin{equation*}
P(z)=1-2z-2z^2+z^3.
\end{equation*}
One easily sees that $P(-1)=0$ and that the greatest positive root of $P(z)$ is given by
$\tau(\Gamma_\star, S_\star)=\frac{3-\sqrt{5}}{2}=\frac{1}{(\varphi-1)^2}$, where $\varphi$ is the golden ratio.
\end{ex}
\vspace{-4mm}
\begin{figure}[htbp]
\centering
\includegraphics[scale=0.5]{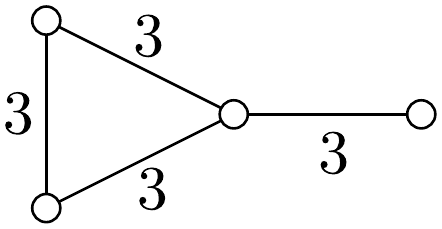}
\caption{The presentation diagram of $(\Gamma_\star,S_\star)$}
\label{fig:figure2.4} 
\end{figure}

\vspace{-2mm}
From now on, we focus on the growth rates of non-spherical, non-affine Coxeter systems.
Typical examples of such Coxeter systems are hyperbolic Coxeter systems.
Three kinds of real algebraic integers appear in the study of the growth rates of hyperbolic Coxeter systems:
Salem numbers, Pisot numbers, and Perron numbers (see \cite[p.84]{Bertin92}).

\vspace{2mm}
An algebraic integer $\tau>1$ of degree at least $4$ is called a \textit{Salem number} if the inverse $\tau^{-1}$ is a Galois conjugate of $\tau$ and the other Galois conjugates lie on the unit circle.
The minimal polynomial of a Salem number is called a \textit{Salem polynomial}.
Parry showed that the growth rates of $2$- and $3$-dimensional cocompact hyperbolic Coxeter systems are Salem numbers \cite{Parry93}.

\vspace{2mm}
An algebraic integer $\tau>1$ is called a \textit{Pisot number} if $\tau$ is an integer or if all of its other Galois conjugates are contained in the unit open disk.
The minimal polynomial of a Pisot number is called a \textit{Pisot polynomial}.
Floyd showed that the growth rates of $2$-dimensional cofinite hyperbolic Coxeter systems are Pisot numbers \cite{Parry93}.
Moreover,
for a $2$-dimensional cofinite hyperbolic Coxeter systems $(\Gamma, S)$,
there exists a sequence of $2$-dimensional cocompact hyperbolic Coxeter systems $(\Gamma_n, S_n)$ whose growth rates $\tau(\Gamma_n, S_n)$ converges to $\tau(\Gamma, S)$ from below.

\vspace{2mm}
An algebraic integer $\tau>1$ is called a \textit{Perron number} if $\tau$ is an integer or if all of its other Galois conjugates are strictly less than $\tau$ in absolute value.
Note that Salem numbers and Pisot numbers are Perron numbers.
The second author showed that the growth rates of $3$-dimensional cofinite hyperbolic Coxeter systems are Perron numbers \cite{Yukita17, Yukita18}.
Note that Komori-Yukita \cite{KomoriYukita15}, and Nonaka-Kellerhals \cite{NonakaKellerhals17} showed that the growth rates of cofinite $3$-dimensional hyperbolic ideal Coxeter systems are Perron numbers.
For a $4$-dimensional cocompact Coxeter system $(\Gamma_P, S_P)$,
Kellerhals and Perren proved that the growth rates are Perron numbers for $\#S_P=5$ and $6$ \cite{RuthPerren11}.
In particular, they conjectured that the growth rates of hyperbolic Coxeter systems are Perron numbers.

\vspace{2mm}
This is a motivation to relate geometric properties of Coxeter systems to the arithmetic nature of their growth rates.
The aim of this paper is to extend the results of Floyd and Parry to \emph{non-spherical, non-affine, and non-hyperbolic} Coxeter systems of dimension at most $2$.

\vspace{2mm}
We use the partial order on the set of Coxeter systems defined by McMullen \cite{McMullen02}.
Let $(\Gamma, S)$ and $(\Gamma', S')$ be Coxeter systems.
Denote $(\Gamma, S)\preceq (\Gamma', S')$ when there exists an injection $\iota:S\to S'$ such that $k(s, t)\leq k'(\iota(s), \iota(t))$,
where $k(s, t)$ and $k'(\iota(s), \iota(t))$ are the orders of $st$ and $\iota(s)\iota(t)$, respectively.
\begin{thm}[Corollary 3.2 \cite{Terragni16}]\label{thm:2.2}
If $(\Gamma, S)\preceq{(\Gamma', S')}$,
then $\tau(\Gamma, S)\leq \tau(\Gamma', S')$.
\end{thm}

For a finitely generated  group $\Gamma$ with ordered finite generating set $S$ with $\#S=N$,
we call the pair $(\Gamma, S)$ a \textit{$N$-marked group}.
Given two $N$-marked groups $(\Gamma, S)$ and $(\Gamma', S')$
we say that they are isomorphic as marked groups when the map $\iota:S\to S'$ sending $s_i$ to $s'_i$ extends to a group isomorphism between $\Gamma$ and $\Gamma'$.
The \emph{space of $N$-marked groups} is the set of isomorphism classes of $N$-marked groups equipped with a metric topology,  given by the Chabauty-Grigorchuk topology; see \cite{Grigorchuk84}.
Let us denote by $\mathcal{C}_N$ the set of marked Coxeter systems of rank $N$.
In \cite{Yukita20},
the second author studied the space $\mathcal{C}_N$ and showed that $\mathcal{C}_N$ is compact.

\begin{thm}[Theorem 3.2, Theorem 3.5 \cite{Yukita20}]\label{thm:2.3}
Let $\{(\Gamma_n, S_n)\}$ and $(\Gamma, S)$ be marked Coxeter systems of rank $N$.
We write $k_{ij}(n)$ (resp. $k_{ij}$) for the order of $s_i(n)s_j(n)$ in $\Gamma_n$ (resp. the order of $s_is_j$ in $\Gamma$).
\begin{itemize}
\item[(1)]
The sequence $\{(\Gamma_n, S_n)\}$ converges to $(\Gamma, S)$ if and only if $\displaystyle \lim_{n\to \infty}k_{ij}(n)=k_{ij}$ for $1\leq i,j \leq N$.
\item[(2)]
If $\displaystyle \lim_{n\to \infty}(\Gamma_n, S_n)=(\Gamma, S)$,
then $\displaystyle \lim_{n\to \infty}\tau(\Gamma_n, S_n)=\tau(\Gamma, S)$.
\end{itemize}
\end{thm}

\section{Growth rates of Coxeter systems with vanishing Euler characteristic}\label{Section3}
Let $(\Gamma, S)$ be a Coxeter system of dimension at most $2$ such that $\chi(L(\Gamma, S))=0$,
where $L(\Gamma, S)$ denotes the geometric realization of its nerve.
In this section, we prove that the growth rate $\tau(\Gamma, S)$ is a Salem number.

\vspace{2mm}
We write $N$ (resp. $E$) for the number of vertices (resp. edges) of the presentation diagram $X(\Gamma, S)$.
Recall that the Euler characteristic of a graph is the number of vertices minus the number of edges.
Since the dimension of $(\Gamma, S)$ is at most $2$,
the underlying graph of $X(\Gamma, S)$ coincides with $L(\Gamma, S)$, and hence $N=E$.
Suppose that the set of labels of the edges of $X(\Gamma, S)$ is $\{k_1,\ldots, k_r\}$.
Let us denote by $E_i$ the number of edges of $X(\Gamma, S)$ labeled by $k_i$.
We obtain the following equality by Steinberg's formula \eqref{eq:2-1}; see also \cite[p.413]{Parry93}.
\begin{align*}
\dfrac{1}{f_{(\Gamma, S)}(z^{-1})}
&=1-\dfrac{N}{[2]}+\sum_{i=1}^r \dfrac{E_i}{[2, k_i]}
=1-\dfrac{E_1+\cdots+E_r}{[2]}+\sum_{i=1}^r \dfrac{E_i}{[2, k_i]}\\
&=1+\sum_{i=1}^r \dfrac{E_i}{[2]}\left(\dfrac{1}{[k_i]}-1\right)
=1+\sum_{i=1}^r \dfrac{E_i}{[2]}\left(\dfrac{z-1}{z^{k_i}-1}-1\right)\\
&=1+\sum_{i=1}^rE_i\dfrac{z-z^{k_i}}{(z+1)(z^{k_i}-1)}.
\end{align*}
Hence
\begin{equation}
\dfrac{z+1}{(z-1) f_{(\Gamma, S)}(z^{-1})}=\dfrac{z+1}{z-1}+\sum_{i=1}^r E_i\dfrac{z-z^{k_i}}{(z-1) (z^{k_i}-1)}.\label{eq:3-2}
\end{equation}

The following lemma is fundamental for the proof.

\begin{lem}[Corollary 1.8 \cite{Parry93}]\label{lem:3.1}
Given integers $k_1,\ldots, k_r\geq 2$ and $E_1,\ldots, E_r\geq 1$ suppose that 
\begin{equation}
\sum_{i=1}^r \left(1-\dfrac{1}{k_i}\right)E_i>2.\label{eq:3-1}
\end{equation}
Let $R(z)$ be the rational function defined by
\begin{equation*}
R(z)=\dfrac{z+1}{z-1}+\sum_{i=1}^r E_k\dfrac{z-z^{k_i}}{(z-1)(z^{k_i}-1)}.
\end{equation*}
Suppose that $P(z)$ and $Q(z)$ are relatively prime monic polynomials with integer coefficients such that $R(z)=P(z)/Q(z)$.
Then,
$P(z)$ and $Q(z)$ are monic polynomials and equal degrees,
and $P(z)$ is a product of distinct irreducible cyclotomic polynomials and exactly one Salem polynomial.
\end{lem}

\begin{thm}\label{thm:3.2}
Let $(\Gamma, S)$ be a non-spherical, non-affine Coxeter system of dimension at most $2$.
If $\chi(L(\Gamma, S))=0$,
then the growth rate $\tau(\Gamma, S)$ is a Salem number.
\end{thm}
\begin{proof}
We apply Lemma \ref{lem:3.1} to \eqref{eq:3-2}.
The proof is divided into three cases:
the cases $N=3$, $N=4$, and $N\geq 5$.

(i) Assume $N=3$.
By assumption,
we have $N=E=3$,
and hence the presentation diagram of $X(\Gamma, S)$ is as in Figure \ref{fig:figure3.1}.
\begin{figure}[htbp]
\centering
\includegraphics[scale=0.5]{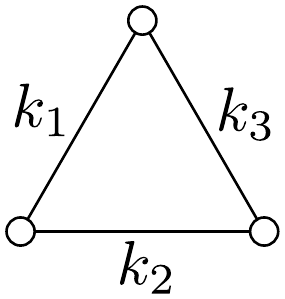}
\caption{The presentation diagram in the case $N=3$}
\label{fig:figure3.1}
\end{figure}

Since $(\Gamma, S)$ is non-spherical and non-affine,
we obtain that
\begin{equation*}
\dfrac{1}{k_1}+\dfrac{1}{k_2}+\dfrac{1}{k_3}<1.
\end{equation*}
Therefore,
\begin{equation*}
\left(1-\dfrac{1}{k_1}\right)+\left(1-\dfrac{1}{k_2}\right)+\left(1-\dfrac{1}{k_3}\right)>2.
\end{equation*}

\vspace{2mm}
(ii) Assume $N=4$.
The presentation diagram $X(\Gamma, S)$ is one of the diagrams in Figure \ref{fig:figure3.2}
We show that one of the labels of $X(\Gamma, S)$ is at least $3$.
\begin{figure}[htbp]
\centering
\includegraphics[scale=0.5]{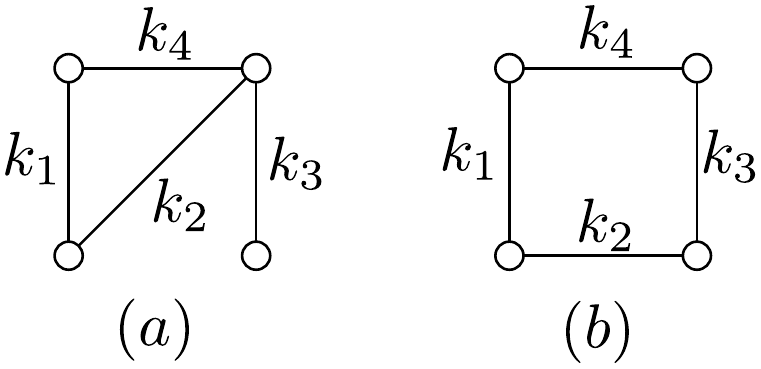}
\caption{The presentation diagrams in the case $N=4$}
\label{fig:figure3.2}
\end{figure}

Suppose that $X(\Gamma, S)$ is the diagram $(a)$.
If $k_1=k_2=k_4=2$,
then the vertices of the triangle generates a spherical parabolic subgroup of $\Gamma$ of rank $3$.
This contradicts our assumption that the dimension of $(\Gamma, S)$ is at most $2$.
Therefore,
one of the labels is at least $3$.
Suppose that $X(\Gamma, S)$ is the diagram $(b)$.
If $k_1=k_2=k_3=k_4=2$,
then the Coxeter diagram $\Cox{(\Gamma, S)}$ is made of two connected components $\widetilde{A}_ 1$ (see Figure \ref{fig:figure2.2} for $\widetilde{A}_1$).
This is a contradiction to the fact that $(\Gamma, S)$ is non-spherical and non-affine.
Therefore,
one of the labels is at least $3$.
Hence
\begin{equation*}
\left(1-\dfrac{1}{k_1}\right)+\left(1-\dfrac{1}{k_2}\right)+\left(1-\dfrac{1}{k_3}\right)+\left(1-\dfrac{1}{k_4}\right)\geq 3\left(1-\dfrac{1}{2}\right)+\left(1-\dfrac{1}{3}\right)>2.
\end{equation*}

(iii) Assume $N\geq 5$.
It follows that
\begin{equation*}
\sum_{i=1}^r \left(1-\dfrac{1}{k_i}\right)E_i=\sum_{i=1}^r E_i-\sum_{i=1}^r \dfrac{E_i}{k_i}=N-\sum_{i=1}^r \dfrac{E_i}{k_i}\geq N-\sum_{i=1}^r\dfrac{E_i}{2}=\dfrac{N}{2}\geq \dfrac{5}{2}>2.
\end{equation*}
Therefore,
\eqref{eq:3-1} holds, and the assertion follows from Lemma \ref{lem:3.1}.
\end{proof}
For later use,
we show the following.
\begin{lem}\label{lem:3.3}
Let $(\Gamma, S)$ be a non-spherical, non-affine Coxeter system of dimension at most $2$.
Suppose that the growth series $f_{(\Gamma, S)}(z)$ satisfies the following equality.
\begin{equation*}
\dfrac{1}{f_{(\Gamma, S)}(z^{-1})}=\dfrac{P(z)}{[2,k_1,\ldots,k_r]},
\end{equation*}
where $P(z)$ is a monic polynomial with integer coefficients.
If $\chi(L(\Gamma, S))=0$,
then $P(z)$ is a product of cyclotomic polynomials and exactly one Salem polynomial.
\end{lem}
\begin{proof}
As in the proof of Theorem \ref{thm:3.2},
we apply Lemma \ref{lem:3.1} to \eqref{eq:3-2}:
\begin{equation*}
\dfrac{z+1}{(z-1)f_{(\Gamma, S)}(z^{-1})}=\dfrac{P_0(z)}{Q_0(z)},
\end{equation*}
where $P_0(z)$ and $Q_0(z)$ are the relatively prime polynomials with integer coefficients satisfying the followings.
\begin{itemize}
\item[(i)] The polynomials $P_0$ and $Q_0$ have same degree $d$.
\item[(ii)] The polynomial $P_0$ is a product of distinct irreducible cyclotomic polynomials and exactly one Salem polynomial.
\end{itemize}
By assumption,
we have 
\begin{equation}
\dfrac{P(z)}{[2,k_1,\ldots,k_r]}=\dfrac{(z-1)P_0(z)}{(z+1)Q_0(z)}.\label{eq:3-3}
\end{equation}
Since every factor of the polynomial $[2, k_1,\ldots,k_r]$ is a cyclotomic polynomial,
the equality \eqref{eq:3-3} implies that $P(z)$ is a product of cyclotomic polynomials and exactly one Salem polynomial.
\end{proof}

\section{Growth rates of Coxeter systems with positive Euler characteristic}\label{Section4}
Let $(\Gamma, S)$ be a Coxeter system of dimension at most $2$ such that $\chi(L(\Gamma, S))\geq 1$,
where $L(\Gamma, S)$ denotes the geometric realization of its nerve.
Recall that $\chi(L(\Gamma, S))$ equals the Euler characteristic of the underlying graph of $X(\Gamma, S)$.
In this section, we prove that the growth rate $\tau(\Gamma, S)$ is a Pisot number.

\begin{lem}\label{lem:4.1}
Let $(\Gamma, S)$ be a non-spherical, non-affine marked Coxeter system of dimension at most $2$ and rank $N$.
Suppose that either the presentation diagram $X(\Gamma, S)$ is disconnected, 
or has an edge labeled by $k\geq 3$.
If $\chi(L(\Gamma, S))\geq 1$,
then there exists a sequence of marked Coxeter systems $\{(\Gamma_n, S_n)\}_{n\geq 7}$ of rank $N$ such that for $n\geq 7$,
the following properties hold.
\begin{itemize}
\item[(1)] $(\Gamma_n, S_n)\preceq{(\Gamma_{n+1}, S_{n+1})}\preceq{(\Gamma, S)}$
\item[(2)] $\dim{(\Gamma_n, S_n)}\leq 2$
\item[(3)] $\chi(L(\Gamma_n, S_n))=\chi(L(\Gamma, S))-1$
\item[(4)] The sequence $\{(\Gamma_n, S_n)\}_{n\geq 7}$ converges to $(\Gamma, S)$ in the space $\mathcal{C}_N$ of marked Coxeter systems of rank $N$.
\end{itemize}
\end{lem}
\begin{proof}
Set $S=\{s_1,\ldots,s_N\}$.
We denote by $E$ and $k_{ij}$ the number of edges of $X(\Gamma, S)$ and the order of the product $s_is_j$, respectively.

\vspace{2mm}
Suppose first that the underlying graph of the presentation diagram $X(\Gamma, S)$ is disconnected.
Let $s_p$ and $s_q$ be two vertices of different connected components of the underlying graph of $X(\Gamma, S)$.
It follows that $k_{pq}=\infty$.
For $n\geq 7$,
we define a marked Coxeter system $(\Gamma_n, S_n)$ of rank $N$ by the following presentation
\begin{equation*}
\Gamma_n=\grouppresentation{s_1(n),\ldots,s_N(n)}{(s_i(n)s_j(n))^{k_{ij}(n)}=1\text{ for }1\leq i, j\leq N},
\end{equation*}
where $k_{ij}(n)=\begin{cases} n & \text{if }\{i, j\}=\{p, q\}\\ k_{ij} & \text{otherwise}\end{cases}$.

We show that $(\Gamma_n, S_n)$ satisfies the desired properties.
For $1\leq i,j\leq N$ and $n\geq 7$, we have
$k_{ij}(n)\leq k_{ij}(n+1)\leq k_{ij}$,
so that 
\begin{equation*}
(\Gamma_n, S_n)\preceq{(\Gamma_{n+1}, S_{n+1})}\preceq{(\Gamma, S)} \ .
\end{equation*}
In order to show that $\dim{(\Gamma_n, S_n)}\leq 2$,
it is sufficient to see that the presentation diagram $X(\Gamma_n, S_n)$ does not contain any of the diagrams depicted in Figure \ref{fig:figure2.3}.
Since $\dim{(\Gamma, S)}\leq 2$,
no such diagram is contained in $X(\Gamma, S)$.
The presentation diagram $X(\Gamma_n, S_n)$ is obtained from $X(\Gamma, S)$ by adding an edge between $s_p$ and $s_q$ labeled by $n$ (see Figure \ref{fig:figure4.1}).
In Figure \ref{fig:figure4.1},
we do not put labels of the edges other than the added edge for simplicity.
Since the vertices $s_p$ and $s_q$ lie in different connected components of the underlying graph of $X(\Gamma, S)$,
every cycle of the underlying graph of $X(\Gamma_n, S_n)$ comes from one of $X(\Gamma, S)$.
Hence we see that $X(\Gamma_n, S_n)$ does not contain any of the diagrams depicted in Figure \ref{fig:figure2.3}.
The Euler characteristics of the underlying graphs of $X(\Gamma_n, S_n)$ and $X(\Gamma, S)$ are equal to $\chi(L(\Gamma_n, S_n))$ and $\chi(L(\Gamma, S))$, respectively.
This observation implies that $\{(\Gamma_n, S_n)\}_{n\geq 7}$ satisfies the property (3).
By definition of $(\Gamma_n, S_n)$,
we have $\displaystyle \lim_{n\to \infty}k_{ij}(n)=k_{ij}$ for $1\leq i, j\leq N$.
Property (1) of Theorem \ref{thm:2.3} implies that $\{(\Gamma_n, S_n)\}_{n\geq 7}$ converges to $(\Gamma, S)$ in $\mathcal{C}_N$.
\begin{figure}[htbp]
\centering
\includegraphics[scale=0.5]{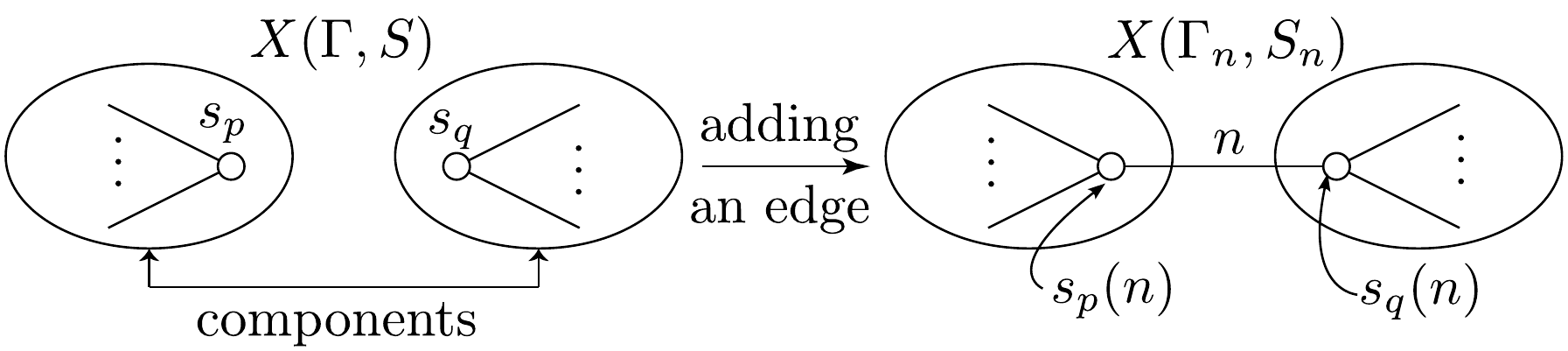}
\caption{Adding an edge between $s_p$ and $s_q$}
\label{fig:figure4.1}
\end{figure}

Suppose next that the underlying graph of $X(\Gamma, S)$ is connected,
and let us show that the underlying graph is a tree.
Since every connected graph with the Euler characteristic $1$ is a tree,
it is sufficient to show that $\chi(L(\Gamma, S))=1$.
By the connectivity of the underlying graph of $X(\Gamma, S)$, 
there exists a spanning tree $T$ of the graph.
We denote by $N_T$ and $E_T$ the number of vertices and of edges of $T$, respectively.
It follows that
$N=N_T$, $E_T\leq E$, and $N_T-E_T=1$.
Since $\chi(L(\Gamma, S))=N-E\geq 1$,
we have
\begin{equation*}
1\leq N-E\leq N-E_T=N_T-E_T=1,
\end{equation*}
and hence $\chi(L(\Gamma, S))=1$.

Since Coxeter systems of rank at most $2$ are spherical or affine,
our assumption implies that $N\geq 3$.
Also by assumption,
there exists an edge $e$ between vertices $s_p$ and $s_q$ of $X(\Gamma, S)$,
labeled by $k_{pq}\geq 3$.
Since the underlying graph of $X(\Gamma, S)$ is a tree with at least $3$ vertices,
we can find an edge $e'$ incident with $e$.
Without loss of generality we can assume that $e$ and $e'$ share the vertex $s_q$.
We write $s_r$ for the endpoint of $e'$ other than $s_q$.
Since the underlying graph of $X(\Gamma, S)$ is a tree,
the vertices $s_p$ and $s_r$ are not joined by an edge.
It follows that $k_{pr}=\infty$.
For $n\geq 7$,
we define a marked Coxeter system $(\Gamma_n, S_n)$ of rank $N$ by the following presentation
\begin{equation*}
\Gamma_n=\grouppresentation{s_1(n),\ldots,s_N(n)}{(s_i(n)s_j(n))^{k_{ij}(n)}=1\text{ for }1\leq i, j\leq N},
\end{equation*}
where $k_{ij}(n)=\begin{cases} n & \text{if }\{i, j\}=\{p, r\}\\ k_{ij} & \text{otherwise}\end{cases}$.

We show that $(\Gamma_n, S_n)$ satisfies the desired properties.
For $1\leq i,j\leq N$ and $n\geq 7$, we have
$k_{ij}(n)\leq k_{ij}(n+1)\leq k_{ij}$,
so that $(\Gamma_n, S_n)\preceq{(\Gamma_{n+1}, S_{n+1})}\preceq{(\Gamma, S)}$.
The presentation diagram $X(\Gamma_n, S_n)$ is obtained from $X(\Gamma, S)$ by adding an edge between $s_p$ and $s_r$ labeled by $n$ (see Figure \ref{fig:figure4.2}).
In Figure \ref{fig:figure4.2},
we do not put labels of the edges other than $3$ edges joining two of $s_p$, $s_q$, and $s_r$ for simplicity.
\begin{figure}[htbp]
\centering
\includegraphics[scale=0.5]{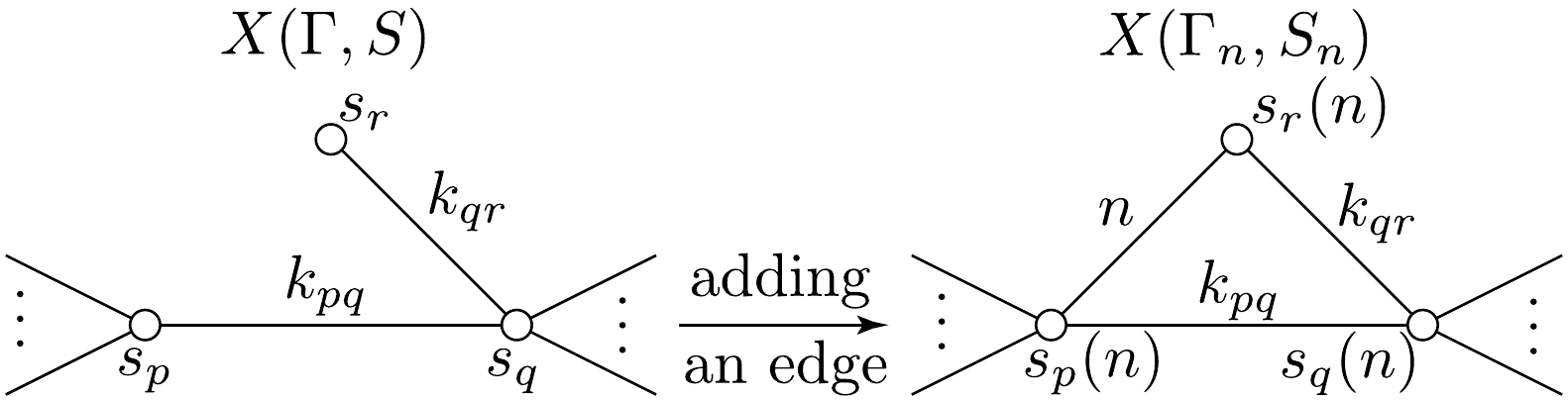}
\caption{Adding an edge between $s_p$ and $s_r$}
\label{fig:figure4.2}
\end{figure}

\noindent
Since the underlying graph of $X(\Gamma, S)$ is a tree,
the one of $X(\Gamma_n, S_n)$ has only one cycle and the cycle consists of $3$ edges joining two of $s_p$, $s_q$, and $s_r$.
Therefore,
the presentation diagram $X(\Gamma_n ,S_n)$ does not contain any of the diagrams in Figure \ref{fig:figure2.3}, which is due to the fact that $k_{pq}\geq 3$ and $n\geq 7$.
It follows that $\dim(\Gamma_n, S_n)\leq 2$.
The same reasoning as before allows to conclude that $\chi(L(\Gamma_n, S_n))=\chi(L(\Gamma, S))-1$.
By definition of $(\Gamma_n, S_n)$,
we have $\displaystyle \lim_{n\to \infty}k_{ij}(n)=k_{ij}$ for $1\leq i, j\leq n$.
Property (1) of Theorem \ref{thm:2.3} implies that $\{(\Gamma_n, S_n)\}_{n\geq 7}$ converges to $(\Gamma, S)$ in $\mathcal{C}_N$.
\end{proof}

\begin{rem}\label{rem:1}
Suppose that $(\Gamma, S)$ is a Coxeter system of at most dimension $2$ and $\chi(L(\Gamma, S))\geq 1$.
By the proof of Lemma \ref{lem:4.1},
one sees that the presentation diagram $X(\Gamma, S)$ is connected if and only if the underlying graph of $X(\Gamma, S)$ is a tree.
Therefore,
if the presentation diagram $X(\Gamma, S)$ is connected and all the edges are labeled by $2$,
then the underlying graph is a tree.
\end{rem}

By using Lemma \ref{lem:4.1} repeatedly,
we obtain the following.
\begin{cor}\label{cor:4.2}
Let $(\Gamma, S)$ be a non-spherical, non-affine marked Coxeter system of dimension at most $2$ and rank $N$.
Suppose that either the presentation diagram $X(\Gamma, S)$ is disconnected, 
or has an edge labeled by $k\geq 3$.
Then,
there exists a sequence of marked Coxeter systems $\{(\Gamma_n, S_n)\}_{n\geq 7}$ of rank $N$ such that for $n\geq 7$,
\begin{itemize}
\item[(1)] $(\Gamma_n, S_n)\preceq{(\Gamma_{n+1}, S_{n+1})}\preceq{(\Gamma, S)}$
\item[(2)] $\dim(\Gamma_n, S_n)\leq 2$
\item[(3)] $\chi(L(\Gamma_n, S_n))=0$
\item[(4)] The sequence $\{(\Gamma_n, S_n)\}$ converges to $(\Gamma, S)$ in the space $\mathcal{C}_N$ of marked Coxeter systems of rank $N$.
\end{itemize}
\end{cor}
\begin{proof}
We take a sequence of marked Coxeter systems $\{(\Gamma_{n_1}, S_{n_1})\}_{n_1\geq 7}$ of rank $N$ as in Lemma \ref{lem:4.1}.
If $\chi(L(\Gamma, S))=1$,
then for $n_1\geq 7$,
\begin{equation*}
\chi(L(\Gamma_{n_1}, S_{n_1}))=\chi(L(\Gamma, S))-1=0.
\end{equation*}
Hence the sequence $\{(\Gamma_{n_1}, S_{n_1})\}_{n_1\geq 7}$ satisfies the properties in Corollary \ref{cor:4.2}

\vspace{1mm}
Suppose that $\chi(L(\Gamma, S))\geq 2$.
The presentation diagram $X(\Gamma_{n_1}, S_{n_1})$ has an edge labeled by $n_1\geq 7$ and $\chi(L(\Gamma_{n_1}, S_{n_1}))=\chi(L(\Gamma, S))-1\geq 1$.
For each $n_1\geq 7$, by applying Lemma \ref{lem:4.1} to $(\Gamma_{n_1}, S_{n_1})$,
there exists a sequence of marked Coxeter systems $\{(\Gamma_{n_1,n_2}, S_{n_1, n_2})\}_{n_2\geq 7}$ of rank $N$ satisfying the properties in Lemma \ref{lem:4.1}.
Moreover,
we may assume that $(\Gamma_{n_1, n_2}, S_{n_1, n_2})\preceq{(\Gamma_{n_1', n_2'}, S_{n_1', n_2'})}$ for $n_1\leq n_1'$ and $n_2\leq n_2'$.
If $\chi(L(\Gamma, S))=2$,
then for $n_1, n_2\geq 7$,
\begin{equation*}
\chi(L(\Gamma_{n_1, n_2}, S_{n_1, n_2}))=\chi(L(\Gamma_{n_1}, S_{n_1}))-1=\chi(L(\Gamma, S))-2=0.
\end{equation*}
Therefore,
the diagonal subsequence $\{(\Gamma_{n,n}, S_{n,n})\}_{n\geq 7}$ satisfies the properties in Corollary \ref{cor:4.2}.
By repeating this procedure until the Euler characteristic vanishes,
which completes the proof.
\end{proof}

Let $(\Gamma, S)$ be a non-spherical, non-affine marked Coxeter system of dimension at most $2$ with $\chi(L(\Gamma, S))\geq 1$.
For simplicity of notation, we write $\chi$ instead of $\chi(L(\Gamma, S))$.
We denote by $N$ (resp. $E$) the number of vertices (resp. edges) of the presentation diagram $X(\Gamma, S)$.
It follows that $N-E=\chi\geq 1$.
Suppose that the set of labels of the edges of $X(\Gamma, S)$ is $\{k_1,\ldots, k_r\}$.
Let us write $E_i$ for the number of edges of $X(\Gamma, S)$ labeled by $k_i$,
so that $E=E_1+\cdots+E_r$.
We obtain the following equality by Steinberg's formula \eqref{eq:2-1}; see also \cite[p.479]{Floyd92}.
\begin{align*}
\dfrac{1}{f_{(\Gamma, S)}(z^{-1})}
&=1-\dfrac{N}{[2]}+\sum_{i=1}^r \dfrac{E_i}{[2, k_i]}
=1-\dfrac{E+\chi}{[2]}+\sum_{i=1}^r \dfrac{E_i}{[2, k_i]}\\
&=\dfrac{[2,k_1,\ldots,k_r]-(E+\chi)[k_1,\ldots,k_r]+\displaystyle \sum_{i=1}^r E_i[k_1,\ldots,\widehat{k_i},\ldots,k_r]}{[2,k_1,\ldots,k_r]}\\
&=\dfrac{[2,k_1,\ldots,k_r]+\displaystyle \sum_{i=1}^r E_i(1-[k_i])[k_1,\ldots,\widehat{k_i},\ldots,k_r]-\chi[k_1,\ldots,k_r]}{[2,k_1,\ldots,k_r]}\\
&=\dfrac{[2,k_1,\ldots,k_r]-\displaystyle \sum_{i=1}^r E_iz[k_i-1][k_1,\ldots,\widehat{k_i},\ldots,k_r]-\chi[k_1,\ldots,k_r]}{[2,k_1,\ldots,k_r]}\\
&=\dfrac{[2,k_1,\ldots,k_r]-\displaystyle \sum_{i=1}^r E_iz[k_1,\ldots,k_i-1,\ldots,k_r]-\chi[k_1,\ldots,k_r]}{[2,k_1,\ldots,k_r]}.
\end{align*}
If $\chi=1$,
then
\begin{align*}
\dfrac{1}{f_{(\Gamma, S)}(z^{-1})}
&=\dfrac{([2]-1)[k_1,\ldots,k_r]-\displaystyle \sum_{i=1}^r E_iz[k_1,\ldots,k_i-1,\ldots,k_r]}{[2,k_1,\ldots,k_r]}\\
&=\dfrac{z\left([k_1,\ldots,k_r]-\displaystyle \sum_{i=1}^r E_i[k_1,\ldots,k_i-1,\ldots,k_r]\right )}{[2,k_1,\ldots,k_r]}.
\end{align*}
We define the polynomial $P(z)$ as 
\begin{equation*}
P(z)=\begin{cases}
[k_1,\ldots,k_r]-\displaystyle \sum_{i=1}^r E_i[k_1,\ldots,k_i-1,\ldots,k_r] & \text{if }\chi=1,\\
[2,k_1,\ldots,k_r]-\displaystyle \sum_{i=1}^r E_iz[k_1,\ldots,k_i-1,\ldots,k_r]-\chi[k_1,\ldots,k_r] & \text{if }\chi\geq 2.
\end{cases}
\end{equation*}
It follows that 
\begin{equation*}
\dfrac{1}{f_{(\Gamma, S)}(z^{-1})}=\begin{cases} 
 \dfrac{zP(z)}{[2,k_1,\ldots,k_r]} & \text{if }\chi=1,\\[10pt]
\dfrac{P(z)}{[2,k_1,\ldots,k_r]} & \text{if }\chi\geq 2. 
\end{cases}
\end{equation*}
In order to show that $P(z)$ is a product of cyclotomic polynomials and exactly one Pisot polynomial,
we use the following.
\begin{lem}\cite[Lemma 1]{Floyd92}\label{lem:4.3}
Let $P(z)$ be a monic polynomial with integer coefficients.
We denote the reciprocal polynomial of $P(z)$ by $\widetilde{P}(z)$, that is,
$\widetilde{P}(z)=z^{\deg{P}} P(z^{-1})$.
Suppose that $P(z)$ satisfies the following conditions.
\begin{itemize}
\item[(i)] $P(0)\neq 0$ and $P(1)<0$
\item[(ii)] $P(z)\neq \widetilde{P}(z)$
\item[(iii)] For sufficiently large integer $m$,
$\dfrac{z^m P(z)-\widetilde{P}(z)}{z-1}$ is a product of cyclotomic polynomials and exactly one Salem polynomial.
\end{itemize}
Then the polynomial $P(z)$ is a product of cyclotomic polynomials and exactly one Pisot polynomial.
\end{lem}

\begin{thm}\label{thm:4.4}
Let $(\Gamma, S)$ be a non-spherical, non-affine Coxeter system of dimension at most $2$ with $\chi(L(\Gamma, S))\geq 1$.
Then the growth rate $\tau(\Gamma, S)$ is a Pisot number.
\end{thm}
\begin{proof}
Assume that $(\Gamma, S)$ has rank $N$.
If the presentation diagram $X(\Gamma, S)$ has no edges,
then by Steinberg's formula \eqref{eq:2-1},
\begin{equation*}
\dfrac{1}{f_{(\Gamma, S)}(z^{-1})}=1-\dfrac{N}{[2]}=\dfrac{z-(N-1)}{[2]}.
\end{equation*}
Therefore,
the growth rate $\tau(\Gamma, S)=N-1$, which is a Pisot number.

From now on,
we assume that the presentation diagram $X(\Gamma, S)$ has at least one edge.
Denote by $E\geq 1$ the number of edges of $X(\Gamma, S)$.
We divide the proof into two cases:
the presentation diagram $X(\Gamma, S)$ is a tree all of whose edges are labeled by $2$,
and otherwise.

\vspace{2mm}
In the first case,
we have $E=N-1$.
Therefore,
by Steinberg's formula \eqref{eq:2-1},
\begin{equation*}
\dfrac{1}{f_{(\Gamma, S)}(z^{-1})}=1-\dfrac{N}{[2]}+\dfrac{N-1}{[2,2]}=\dfrac{[2,2]-N[2]+N-1}{[2,2]}=\dfrac{z(z-(N-2))}{(1+z)^2}.
\end{equation*}
Therefore,
the growth rate $\tau(\Gamma, S)=N-2$, which is a Pisot number.

\vspace{2mm}
In other case,
by Remark \ref{rem:1},
either the presentation diagram $X(\Gamma, S)$ is disconnected or has an edge labeled by $k\geq 3$.
We fix an ordering of the generating set $S$.
Let us take a sequence of marked Coxeter systems $\{(\Gamma_n, S_n)\}_{n\geq 7}$ of rank $N$ as in Corollary \ref{cor:4.2}.
It follows from property (3) that the number of edges of $X(\Gamma_n, S_n)$ equals $E+\chi(L(\Gamma, S))$.
In particular,
for every $n\geq 7$ different from $k_1,\ldots,k_r$,
the number of edges of $X(\Gamma_n, S_n)$ labeled by $n$ is equal to $\chi(L(\Gamma, S))$.
For simplicity,
we write $\chi$ instead of $\chi(L(\Gamma, S))$.
By Steinberg's formula \eqref{eq:2-1},
we have
\begin{equation*}
\dfrac{1}{f_{(\Gamma_n, S_n)}(z^{-1})}
=1-\dfrac{N}{[2]}+\sum_{i=1}^r \dfrac{E_i}{[2, k_i]}+\dfrac{\chi}{[2, n]}=\dfrac{P_n(z)}{[2,k_1,\ldots,k_r,n]},
\end{equation*}
where 
\begin{equation*}
P_n(z)=[2,k_1,\ldots,k_r,n]-N[k_1,\ldots,k_r,n]+\sum_{i=1}^r E_i[k_1,\ldots,\widehat{k_i},\ldots,k_r,n]+\chi[k_1,\ldots,k_r].
\end{equation*}
From the equality $N=E_1+\cdots+E_r+\chi$,
we obtain that 
\begin{align*}
P_n(z)&=[2,k_1,\ldots,k_r,n]-\sum_{i=1}^r E_i z[k_1,\ldots,k_i-1,\ldots,k_r,n]-\chi z[k_1,\ldots,k_r,n-1].
\end{align*}
Define the polynomials $P(z)$ as
\begin{equation*}
P(z)=\begin{cases}
[k_1,\ldots,k_r]-\displaystyle \sum_{i=1}^r E_i[k_1,\ldots,k_i-1,\ldots,k_r] & \text{if }\chi=1,\\
[2,k_1,\ldots,k_r]-\displaystyle \sum_{i=1}^r E_iz[k_1,\ldots,k_i-1,\ldots,k_r]-\chi[k_1,\ldots,k_r] & \text{if }\chi\geq 2,
\end{cases}
\end{equation*}
and $\widetilde{P}(z)=z^{\deg P}P(z^{-1})$.
We obtain that 
\begin{equation*}
(z-1)P_n(z)=
\begin{cases} 
z^nP(z)-\widetilde{P}(z) & \text{if }\chi=1,\\
z^{n+1}P(z)-\widetilde{P}(z) & \text{if }\chi\geq 2.
\end{cases}
\end{equation*}

In order to apply Lemma \ref{lem:4.3} to $P(z)$,
we need to show that $P(0)\neq 0, P(1)<0$, and that $P(z)$ is not reciprocal.
First,
\begin{equation*}
P(0)=\begin{cases}
1-E &\text{if }\chi=1,\\
1-\chi & \text{if }\chi\geq 2.
\end{cases}
\end{equation*}
It follows that $P(0)\neq 0$.
Since $P(z)$ is monic, 
we also conclude that $P(z)$ is not reciprocal.
Finally, we see that $P(1)<0$ as follows.
In the case $\chi=1$,
\begin{equation*}
P(1)
=\prod_{i=1}^r k_i-\sum_{i=1}^r \left(E_i\cdot \prod_{j=1}^r k_j\cdot  \dfrac{k_i-1}{k_i}\right)
=\prod_{i=1}^r k_i\cdot \left\{1-\sum_{i=1}^r E_i\left(1-\dfrac{1}{k_i}\right)\right\}.
\end{equation*}
In the case $\chi\geq 2$,
\begin{align*}
P(1)
&=2\prod_{i=1}^r k_i-\sum_{i=1}^r \left(E_i\cdot \prod_{j=1}^r k_j\cdot  \dfrac{k_i-1}{k_i}\right)-\chi\prod_{i=1}^r k_i\\
&=\prod_{i=1}^r k_i\cdot \left\{2-\sum_{i=1}^r E_i\left(1-\dfrac{1}{k_i}\right)-\chi\right\}\\
&=\prod_{i=1}^r k_i\cdot \left\{(1-\chi)+1-\sum_{i=1}^r E_i\left(1-\dfrac{1}{k_i}\right)\right\}
\end{align*}
For $N\geq 4$,
\begin{equation*}
\sum_{i=1}^r E_i\left(1-\dfrac{1}{k_i}\right)\geq \sum_{i=1}^r E_i\left(1-\dfrac{1}{2}\right)=\dfrac{E}{2}=\dfrac{N-1}{2}\geq \dfrac{3}{2}>1.
\end{equation*}
It follows that $P(1)<0$ from
\begin{equation*}
1-\sum_{i=1}^r E_i\left(1-\dfrac{1}{k_i}\right)<0.
\end{equation*}
For $N=3$,
the presentation diagram $X(\Gamma, S)$ is one of the diagrams in Figure \ref{fig:figure4.3}.
\begin{figure}[htbp]
\centering
\includegraphics[scale=0.5]{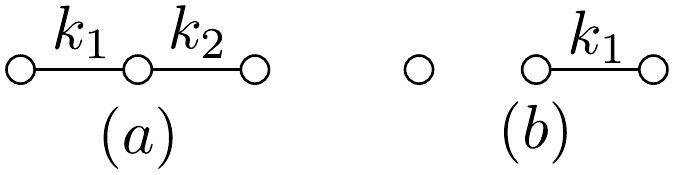}
\caption{The presentation diagrams for $N=3$}
\label{fig:figure4.3}
\end{figure}

\noindent 
In the case $(a)$,
we necessarily have $k_1\geq 3$ or $k_2\geq 3$,
so that 
\begin{equation*}
1-\left(1-\dfrac{1}{k_1}\right)-\left(1-\dfrac{1}{k_2}\right)\leq 1-\dfrac{1}{2}-\dfrac{2}{3}=-\dfrac{1}{6}<0.
\end{equation*}
Hence $P(1)<0$.
The same reasoning applies to the case $(b)$.

By Lemma \ref{lem:4.3},
the polynomial $P(z)$ is a product of cyclotomic polynomials and exactly one Pisot polynomial,
and hence the growth rate $\tau(\Gamma, S)$ is a Pisot number.
\end{proof}

\begin{thm}
Let $(\Gamma, S)$ be a non-spherical, non-affine Coxeter system of dimension at most $2$ with $\chi(L(\Gamma, S))\geq 1$.
Then, there exists a sequence of Coxeter systems $(\Gamma_n, S_n)$ of dimension at most $2$ with vanishing Euler characteristic such that the growth rate $\tau(\Gamma_n, S_n)$ converges to $\tau(\Gamma, S)$ from below.
\end{thm}
\begin{proof}
We denote by $N$ the rank of $(\Gamma, S)$.
The proof is divided into two cases:
either the presentation diagram $X(\Gamma, S)$ is disconnected or has an edge labeled by $k\geq 3$,
and otherwise.

\vspace{2mm}
In the first case,
we fix an ordering of $S$ and 
we take a sequence of marked Coxeter systems $\{(\Gamma_n, S_n)\}_{n\geq 7}$ of rank $N$ as in Corollary \ref{cor:4.2}.
By combining Theorem \ref{thm:2.2}, Theorem \ref{thm:2.3}, and Theorem \ref{thm:3.2},
we conclude that the growth rate $\tau(\Gamma_n, S_n)$ is a Salem number and the sequence $\{\tau(\Gamma_n, S_n)\}_{n\geq 7}$ converges to $\tau(\Gamma, S)$ from below.

\vspace{2mm}
In other case,
by Remark \ref{rem:1},
the presentation diagram $X(\Gamma, S)$ is a tree with all edges labeled by $2$
and the growth rate $\tau(\Gamma, S)=N-2$.
Since $(\Gamma, S)$ is non-spherical and non-affine,
it forces $N\geq 4$.
Consider the marked Coxeter system $(\hat{\Gamma}, \hat{S})$ of rank $N$ whose presentation diagram $X(\hat{\Gamma}, \hat{S})$ is depicted in Figure \ref{fig:figure4.4}.
\begin{figure}[htbp]
\centering
\includegraphics[scale=0.5]{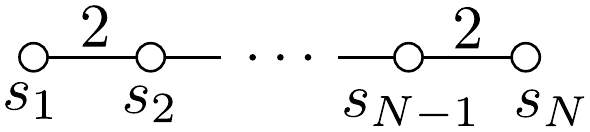}
\caption{The presentation diagram $X(\hat{\Gamma}, \hat{S})$}
\label{fig:figure4.4}
\end{figure}

\noindent
Let $(\Gamma_n, S_n)$ be the marked Coxeter system of rank $N$ whose presentation diagram $X(\Gamma_n, S_n)$ is obtained by adding an edge labeled by $n\geq 3$ between $s_1$ and $s_N$.
As a direct consequence, $(\Gamma_n, S_n)$ converges to $(\hat{\Gamma}, \hat{S})$ in the space  of marked Coxeter systems $\mathcal{C}_N$ of rank $N$.
Since $\tau(\hat{\Gamma}, \hat{S})=\tau(\Gamma, S)$,
by combining Theorem \ref{thm:2.2}, Theorem \ref{thm:2.3}, and Theorem \ref{thm:3.2},
the assertion follows.
\end{proof}

\section{Growth rates of Coxeter systems and Perron numbers}\label{Section5}

In this section,
we consider Coxeter systems of dimension at most $2$ with negative Euler characteristic.  We provide infinite sequences of such Coxeter systems whose growth rates are Perron numbers.  

\vspace{2mm}
The following is fundamental for our considerations. 
Let $(\Gamma_{\star},S_{\star})$ be the Coxeter system with presentation diagram depicted in Figure \ref{fig:figure5.1}. As  discussed in Example \ref{ex:1},   the radius of convergence of  its growth series is given by $r_{\star}=\frac{1}{\tau(\Gamma_{\star}, S_{\star})}=(\varphi-1)^2$, where $\varphi$ is the golden ratio.  
\begin{figure}[htbp]
\centering
\includegraphics[scale=0.5]{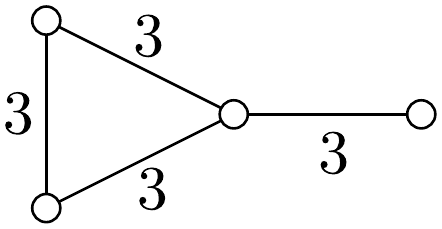}
\caption{The presentation diagram of $(\Gamma_\star,S_\star)$}
\label{fig:figure5.1} 
\end{figure}

\vspace{2mm}
We study infinite sequences of  Coxeter systems $(\Gamma,S)$ such that $(\Gamma_{\star},S_{\star})\preceq(\Gamma,S)$, see Section \ref{Section2.3}.  Beside that, we restrict  their presentation diagrams as follows.  
\begin{thm}
Let $(\Gamma_{k, N}, S)$ be a non-spherical, non-affine Coxeter system of dimension at most $2$ and rank $N$, such that all labels of the presentation diagram $X(\Gamma_{k,N}, S)$ are the same $k\geq 3$.
Denote by $E$ the number of edges of $X(\Gamma_{k, N},S)$. 

If $(\Gamma_{k, N},S)$ satisfies the following properties:
\begin{itemize}
\item[$(i)$] $(\Gamma_\star, S_\star)\preceq{(\Gamma_{k, N}, S)}$
\item[($ii$)]
$E=a(N-1)$  for a rational number $1 < a \leq \frac{(1+\varphi)^2}{3}$
\end{itemize}
Then, the growth rate $\tau(\Gamma_{k, N}, S)$ is a Perron number.
\label{thm:5.1}
\end{thm}

\begin{figure} [htbp]
\centering
\includegraphics[scale=0.4]{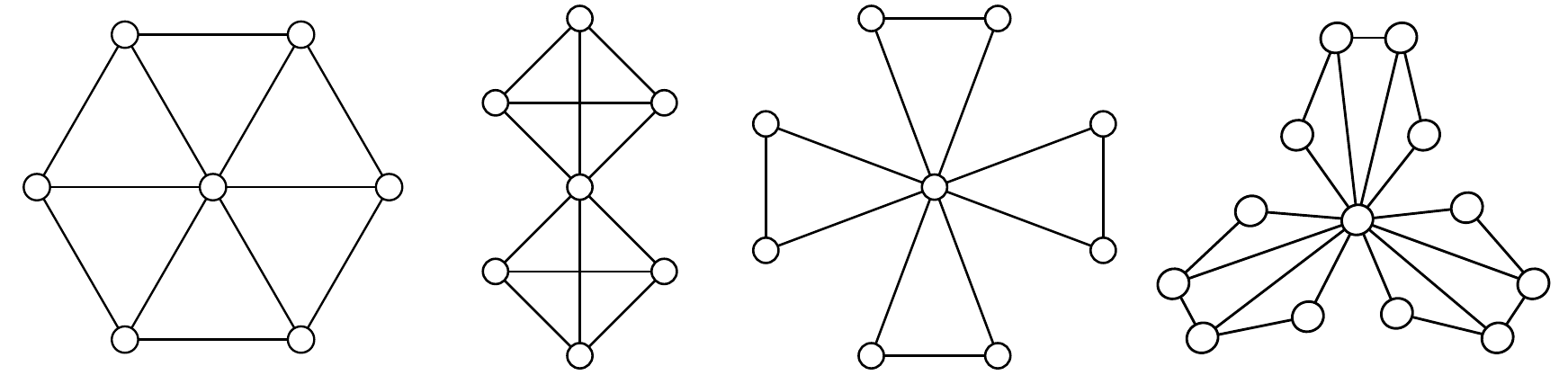}
\caption{The graphs $W_{7}$, $\mathcal W(4,2)$,  $F_{4}=\mathcal W(3,4)$, and $\mathcal T(5,3)$}
\label{fig:5.2}
\end{figure}

There exists a variety of infinite sequences of Coxeter systems satisfying the hypothesis of Theorem \ref{thm:5.1}.
We provide some examples in Figure \ref{fig:5.2} in terms of the underlying graphs of their presentation diagrams.  For terminolgy, we refer to \cite{Gallian98}.  Such Coxeter systems all satisfy $\chi\leq-1$.  
For instance,
the family of \textit{wheel graphs} $W_{N}$,  for all $N\geq 6$,
formed by a cycle of length $N-1$ and a universal vertex,
that is,
a central vertex linked to each other vertex.
In that case the number of edges of the graph is given by $E=2(N-1)$.
Same goes for the \textit{windmill graphs} of type $\mathcal W(4,l)$,
$l\geq 2$,
made of $l$ copies of complete graphs $K_4$ joined at common central vertex.
Also fitting the hypothesis of the Theorem is the family of \textit{friendship graphs} $F_{l}=\mathcal W(3,l)$ for $l\geq 3$,
for which $E=\frac{3}{2}(N-1)$.
Several variations of those graphs can be constructed.
For example,
we defined the \textit{triangulated bouquet} $\mathcal T(k,l)$ as the graph formed by  $l$ copies of $k$-cycles glued in a common vertex $v$,
such that any other vertex is linked to $v$.
In this case,
$v$ is universal and one has $E=\frac{2k-1}{k-1}(N-1)$.

\begin{rem}
A Coxeter system is said to be $\infty$-spanned if there exists a spanning tree of its Coxeter diagram with edges labelled $\infty$ only.  In \cite{KolpakovTalambutsa19},  Kolpakov and Talambutsa proved that the growth rate of  $\infty$-spanned Coxeter systems are Perron numbers.  We mention that, by the existence of an universal vertex in the presentation diagram of the Coxeter systems  discussed above,  such a spanning tree cannot be found in the corresponding Coxeter diagram. 
\end{rem}

\begin{proof}
Assume that $(\Gamma_{k,N},S)$ satisfies the hypothesis of Theorem \ref{thm:5.1}.  In what follows, we denote by $f_{k,N}(z)=\frac{Q_{k,N}(z)}{P_{k,N}(z)}$ 
the growth series of $(\Gamma_{k,N}, S)$,  by $ r_{k, N}$ its radius of convergence, and by $\tau_{k, N}$ the growth rate of $(\Gamma_{k,N}, S)$ respectively. 
 Recall that $r_{k,N}$ is the smallest positive real root of $P_{k,N}(z)$.  
In order to prove that $\tau_{k,N}$ is a Perron number,  we show that $r_{k,N}$ is the unique zero of $P_{k,N}(z)$ with smallest modulus.  

\vspace{2mm}
By Steinberg's formula \eqref{eq:2-1},
\begin{equation*}
\frac{1}{f_{k,N}(z^{-1})}=1-\frac{N}{[2]}+\frac{a(N-1)}{[2,k]}=\frac{[2,k]-N[k]+a(N-1)}{[2,k]} \ .
\end{equation*}
Therefore, the the denominator of $f_{k,N}(z)$ is given by 
\begin{equation*}
{P_{k,N}(z)}=1+(2-N)(z+z^2+\dots+z^{k-1})+(a-1)(N-1)z^k \ .
\end{equation*}
We write ${P_{k,N}(z)}=h_N(z)+R_{k,N}(z)$
where $h_N(z)$ is the quadratic polynomial 
\begin{equation*}
h_{N}(z)= 1+(2-N)(z+z^2)\, 
\end{equation*}
and $R_{k,N}(z)$ is the remaining part.  By hypothesis $(\Gamma_{\star},S_{\star})\preceq{(\Gamma_{k,N},S)}$, therefore by Theorem \ref{thm:2.3},
we conclude 
\begin{equation}
\tau(\Gamma_{\star}, S_{\star}) \leq \tau_{k,N} \ .
\label{eq:5-1}
\end{equation}
It follows that the associated radii of convergence satisfy $r_{k,N}\leq r_{\star}$.

\vspace{2mm}
In order to prove that $r_{k,N}$ is the unique root with smallest modulus of $P_{k,N}(z)$, we use Rouch{\'{e}}'s theorem on the open disk $D(0,r_{\star})$. We first observe that  $h_N(z)$ has a unique root in $D(0,r_{\star})$,  and we prove $\abs{h_{N}(z)}-\abs{R_{k,N}}>0$ on $|z|=r_{\star}$. 

\vspace{2mm}
An easy analysis of the roots shows that for any $N$, $h_N(z)$ admits a unique root in the open disk $D(0, r_{\star})$.
Moreover,  on the circle $|z|=r_{\star}$, one has 
\begin{equation}
 \abs{h_{N}(z)}\geq   
\begin{cases}\abs{1+(2-N)(r_{\star}+r_{\star}^2)} &  \text{ if } N\leq 8 \\ \abs{1+(2-N)(r_{\star}^2-r_{\star})} & \text{ if }N\geq 9 
\end{cases} \ .
\label{eq:5-2}
\end{equation}
This is proved by seeing that the minimum of $\abs{h_{N}(r_{\star}e^{i\theta})}$ is achieved at $\theta = 2\pi$ for $N\leq 8$,  and at $\theta = \pi$ for $N\geq 9$.

\vspace{2mm}
We assume that $N\geq 9$.  
For the case where $ N\leq 8$,  one can complete the proof by applying the same reasoning. 

\vspace{2mm}
\noindent
Let $z$  be such that $\abs{z}=r_{\star}$, and put  $\Delta_{k,N}(z)=
\abs{h_{k,N}(z)}-\abs{R_{k,N}(z)}.$ 
Since $a>1$, by the triangle inequality, 
 \begin{equation*}
\abs{R_{k,N}(z)}\leq (N-2)(r_{\star}^3+\dots+r_{\star}^{k-1})+(a-1)(N-1)r_{\star}^{k} \ .
 \end{equation*}
Also, by \eqref{eq:5-2},  one has 
$
 \abs{h_{N}(z)} \geq \abs{1+(2-N)(r_{\star}^2-r_{\star})} \geq 1+(2-N)(r_{\star}^2-r_{\star}).$
It follows that  
\begin{equation}
\Delta_{k,N}(z)\geq  N\left(1+2r-\frac{1-r_{\star}^{k}}{1-r_{\star}}-(a-1)r_{\star}^k\right)-3-4r_{\star}+2\frac{1-r_{\star}^k}{1-r_{\star}} +(a-1)r_{\star}^k \ .
 \label{eq:5-4}
\end{equation}
Denote the right-hand-side of \eqref{eq:5-4} by $\Lambda_{k,N}=N\alpha_k+\beta_k$, that is \begin{equation*} 
\alpha_k= 1+2r-\frac{1-r_{\star}^{k}}{1-r_{\star}}-(a-1)r_{\star}^k  \ ,  \quad \beta_k = -3-4r_{\star}+2\frac{1-r_{\star}^k}{1-r_{\star}} +(a-1)r_{\star}^k \ . 
\end{equation*}
Then we have
\begin{equation*}
\alpha_k-\alpha_{k+1}= r_{\star}^k(2-a+(a-1)r_{\star})>0 \ ,
\end{equation*}
so that $\alpha_k$ decreases with respect to $k$, by hypothesis on $a$.
Hence \begin{equation*}
\alpha_{k}
\geq  \lim\limits_{k\rightarrow\infty} \alpha_{k} = 1+2r_{\star}-\frac{1}{1-r_{\star}}=\frac{7-3\sqrt 5}{2}>0 \ .
\end{equation*} 
This precisely means that $\Lambda_{k,N}$ increases with respect to $N$.  
In addition, we get $ \Lambda_{k,N}-\Lambda_{k+1,N} >0$ for $N >\lambda_a$, where
\begin{equation*}
\lambda_a=\frac{3-a+(a-1)r_{\star}}{2-a+(a-1)r_{\star}} \ .
\end{equation*}
Therefore,  for each  $a\leq \frac{(1+\varphi)^2}{3}$,  we find that $\Lambda_{k,N}$ decreases with respect to $k$ for all $N\geq 9$.  Consequently, 
\begin{align*}
 \Delta_{k,N}(z) \geq  \lim\limits_{k\rightarrow\infty} \Lambda_{k,N} =
 N\left(1+2r_{\star}-\frac{1}{1-r_{\star}} \right)-3-4r_{\star}+\frac{2}{1-r_{\star}} \, ,
 \end{align*}
which leads to  $ \Delta_{k,N} (z)>0$ when $N>  \frac{3+4r_{\star}-2\frac{1}{1-r_{\star}}}{1+2r_{\star}-\frac{1}{1-r_{\star}}} = \frac{11+\sqrt{45}}{2}$.  This is true for any $N\geq 9$, which finishes the proof. 
 \end{proof}

 We proved in  Theorem \ref{thm:3.2} and Theorem \ref{thm:4.4} that growth rates of Coxeter systems of dimension at most $2$ with positive and vanishing Euler characteristic are Salem and Pisot numbers respectively.    By Theorem \ref{thm:5.1},  the growth rates of infinitely many Coxeter systems with negative Euler characteristic are Perron numbers.  Inspired by these observations,  we make the following claim.
\begin{conj}
\textit{The growth rate of any Coxeter system of dimension at most $2$ is a Perron number.}
\end{conj}

\noindent
Note that,  there exist  Coxeter systems of dimension at most $2$ such that $\chi\leq-1$ whose growth rates are Perron numbers but  {\it  are neither} Pisot numbers nor Salem numbers.
For instance, the $3$-dimensional hyperbolic \textit{ideal} Coxeter system $(\Gamma_0, S_0)$ whose presentation diagram admits labels $3$ only and is depicted in Figure \ref{fig:figure5.3}.
\begin{figure}[htbp]
\centering
\includegraphics[scale=0.5]{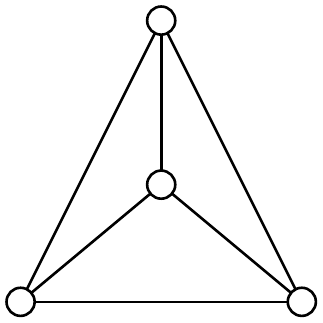}
\caption{The presentation diagram of $(\Gamma_0, S_0)$}
\label{fig:figure5.3}
\end{figure}

\subsection*{Acknowledgement}
The authors would like to express their gratitude to Professor Ruth Kellerhals for helpful discussions.
The second  author is supported by JSPS Grant-in-Aid for Early-Career Scientists Grant Number JP20K14318.

\bibliographystyle{plain}
\bibliography{reference.bib}

\begin{thebibliography}{10}

\bibitem{Bertin92}
M.J. Bertin, A.~Decomps-Guilloux, M.~Grandet-Hugot, M.~Pathiaux-Delefosse, and
  J.P. Schreiber.
\newblock {\em Pisot and Salem Numbers}.
\newblock Birkh{\"{a}}user Basel, 1992.

\bibitem{Davis12}
Michael Davis.
\newblock {\em The Geometry and Topology of Coxeter Groups. (LMS-32)}.
\newblock Princeton University Press, 2012.

\bibitem{FeliksonTumarkin10}
Anna Felikson and Pavel Tumarkin.
\newblock Reflection subgroups of {C}oxeter groups.
\newblock {\em Trans. Amer. Math. Soc.}, 362(2):847--858, 2010.

\bibitem{Floyd92}
William~J. Floyd.
\newblock Growth of planar {C}oxeter groups, {P}.{V}. numbers, and {S}alem
  numbers.
\newblock {\em Math. Ann.}, 293(3):475--483, 1992.

\bibitem{Gallian98}
Joseph~A. Gallian.
\newblock {\em A dynamic survey of graph labeling}.
\newblock Number~6 in Dynamic Survey. Electron. J. Combin., twenty-fourth
  edition, 2021.

\bibitem{Grigorchuk84}
Rostislav~I. Grigorchuk.
\newblock Degrees of growth of finitely generated groups and the theory of
  invariant means.
\newblock {\em Izv. Akad. Nauk SSSR Ser. Mat.}, 48(5):939--985, 1984.

\bibitem{Gromov81}
Mikhael Gromov.
\newblock Groups of polynomial growth and expanding maps.
\newblock {\em Inst. Hautes \'{E}tudes Sci. Publ. Math.}, 53:53--73, 1981.

\bibitem{Harpe87}
Pierre de~la Harpe.
\newblock Groupes de {C}oxeter infinis non affines.
\newblock {\em Exposition. Math.}, 5(1):91--96, 1987.

\bibitem{Humphreys90}
James~E. Humphreys.
\newblock {\em Reflection Groups and Coxeter Groups}.
\newblock Cambridge Studies in Advanced Mathematics. Cambridge University
  Press, 1990.

\bibitem{RuthPerren11}
Ruth Kellerhals and Genevi\`eve Perren.
\newblock On the growth of cocompact hyperbolic {C}oxeter groups.
\newblock {\em European J. Combin.}, 32(8):1299--1316, 2011.

\bibitem{Kolpakov12}
Alexander Kolpakov.
\newblock Deformation of finite-volume hyperbolic {C}oxeter polyhedra, limiting
  growth rates and {P}isot numbers.
\newblock {\em European J. Combin.}, 33(8):1709--1724, 2012.

\bibitem{KolpakovTalambutsa19}
Alexander Kolpakov and Alexey Talambutsa.
\newblock Growth rates of coxeter groups and perron numbers.
\newblock {\em arXiv:1912.05608, accepted for publication in Int. Math. Res.
  Not.}, 2021.

\bibitem{KomoriYukita15}
Yohei Komori and Tomoshige Yukita.
\newblock On the growth rate of ideal {C}oxeter groups in hyperbolic 3-space.
\newblock {\em Proc. Japan Acad. Ser. A Math. Sci.}, 91(10):155--159, 2015.

\bibitem{McMullen02}
Curtis~T. McMullen.
\newblock Coxeter groups, {S}alem numbers and the {H}ilbert metric.
\newblock {\em Publ. Math. Inst. Hautes \'{E}tudes Sci.}, 95:151--183, 2002.

\bibitem{Munkres84}
James~R Munkres.
\newblock {\em Elements of {A}lgebraic {T}opology}.
\newblock CRC press, first edition, 1984.

\bibitem{NonakaKellerhals17}
Jun Nonaka and Ruth Kellerhals.
\newblock The growth rates of ideal {C}oxeter polyhedra in hyperbolic 3-space.
\newblock {\em Tokyo J. Math.}, 40(2):379--391, 2017.

\bibitem{Parry93}
Walter Parry.
\newblock Growth series of {C}oxeter groups and {S}alem numbers.
\newblock {\em J. Algebra}, 154(2):406--415, 1993.

\bibitem{Ratcliffe19}
John~G. Ratcliffe.
\newblock {\em {F}oundations of {H}yperbolic {M}anifolds}.
\newblock Graduate {T}exts in {M}athematics. Springer Cham, third edition,
  2019.

\bibitem{Solomon66}
Louis Solomon.
\newblock The orders of the finite {C}hevalley groups.
\newblock {\em J. Algebra}, 3:376--393, 1966.

\bibitem{Steinberg68}
Robert Steinberg.
\newblock {\em Endomorphisms of linear algebraic groups}.
\newblock Number~80 in Memoirs of the {A}merican {M}athematical {S}ociety.
  American {M}athematical {S}ociety, 1968.

\bibitem{Terragni16}
Tommaso Terragni.
\newblock On the growth of a {C}oxeter group.
\newblock {\em Groups Geom. Dyn.}, 10(2):601--618, 2016.

\bibitem{Yukita17}
Tomoshige Yukita.
\newblock On the growth rates of cofinite 3-dimensional hyperbolic {C}oxeter
  groups whose dihedral angles are of the form {$\frac\pi m$} for
  {$m=2,3,4,5,6$}.
\newblock In {\em Geometry and analysis of discrete groups and hyperbolic
  spaces}, RIMS K\^{o}ky\^{u}roku Bessatsu, B66, pages 147--165. Res. Inst.
  Math. Sci. (RIMS), Kyoto, 2017.

\bibitem{Yukita18}
Tomoshige Yukita.
\newblock Growth rates of 3-dimensional hyperbolic {C}oxeter groups are
  {P}erron numbers.
\newblock {\em Canad. Math. Bull.}, 61(2):405--422, 2018.

\bibitem{Yukita20}
Tomoshige Yukita.
\newblock On the continuity of the growth rate on the space of coxeter systems.
\newblock {\em arXiv:2011.09953}, 2020.

\end{thebibliography}

\end{document}